\documentclass{amsart}
\usepackage{amsfonts,amssymb,amscd,amsmath,enumerate,verbatim,calc}
\usepackage[all]{xy}
\usepackage{mathtools}

\newcommand{\bP}{\mathbb{\partial}}
\newcommand{\cB}{\mathcal{B}}

\newcommand{\m}{\mathfrak{m} }

\newcommand{\Bp}{\mathbf{\partial} }
\newcommand{\Bx}{\mathbf{X} }

\newcommand{\Z}{\mathbb{Z} }

\newcommand{\rt}{\rightarrow}

\newcommand{\Supp}{\operatorname{Supp}}

\newcommand{\Ass}{\operatorname{Ass}}

\theoremstyle{plain}

\newtheorem{theorem}{Theorem}[section]
\newtheorem{corollary}[theorem]{Corollary}
\newtheorem{lemma}[theorem]{Lemma}
\newtheorem{proposition}[theorem]{Proposition}

\theoremstyle{definition}
\newtheorem{definition}[theorem]{Definition}
\newtheorem{s}[theorem]{}
\newtheorem{remark}[theorem]{Remark}

\theoremstyle{remark}

\begin{document}

\title[\MakeLowercase{de} Rham]{On a Relation between Euler characteristics of \MakeLowercase{de} Rham cohomology and Koszul cohomology of graded local cohomology modules}

\author{Tony~J.~Puthenpurakal}
\author{Rakesh B. T. Reddy  }
\date{\today}
\address{Department of Mathematics, IIT Bombay, Powai, Mumbai 400 076}

\email{tputhen@gmail.com}

\address{Department of Mathematics, Faculty of Science and Technology(IcfaiTech), ICFAI Foundation for Higher Educaion, Hyderabad}
\email{rakesht@ifheindia.org}

\begin{abstract}

Let $K$ be a  field of characteristic zero.
Let $R = K[X_0, X_1,\ldots,X_n]$ be standard graded.  Let
$A_{n+1}(K)$ be the $(n + 1)^{th}$ Weyl algebra over $K$.
Let $I$ be a homogeneous ideal of $R$ and let $M  = H^i_I(R)$ for some $i \geq 0$.
 By a result of
Lyubeznik, $M$ is a graded holonomic
$A_{n +1}(K)$-modules for each $i \geq 0$. Let $\chi^c(\Bp, M)$ ($\chi^c(\Bx, M)$) be the Euler characteristics of de Rham cohomology (resp. Koszul cohomology) of $M$. We  prove
$\chi^c(\Bp, M) = (-1)^{n+1}\chi^c(\Bx, M)$.
\end{abstract}

\maketitle
\section{Introduction}
Let $K$ be a field of characteristic zero and let $R
= K[X_0, X_1,\ldots,X_n]$. Let $A_{n + 1}(K)  = K<X_0, X_1,\ldots,X_n, \partial_0, \partial_1, \ldots,
\partial_n>$  be the $(n +1)^{th}$ Weyl algebra over $K$.
Let $N$ be a left $A_{n+1}(K)$ module. Now $\bP =
\partial_0, \partial_1,\ldots,\partial_n$ are pairwise commuting $K$-linear
maps. So we can consider the de Rham complex $K(\bP;N)$. Notice that
the de Rham cohomology modules  $H^*(\bP;N)$ are in general only
$K$-vector spaces. They are finite dimensional if $N$ is holonomic;
see \cite[Chapter 1, Theorem 6.1]{BJ}. In particular the Euler
characteristic $\chi^c(\Bp,N) = \sum_{i = 0}^{n+1} (-1)^i\dim_K H^i(\bP; N)$ is
a finite number.
In \cite{BJ} holonomic $A_{n+1}(K)$ modules are denoted as $\cB_{n +1}(K)$, the  \textit{Bernstein} class of left $A_{n+1}(K)$ modules.
 Similarly if $N$ is a holonomic $A_{n+1}(K)$-module then the Koszul cohomology modules $H^i(\Bx, N)$ are finite dimensional $K$-vector spaces. So again the Euler characteristic
$\chi^c(\Bx, N) = \sum_{i = 0}^{n+1} (-1)^i\dim_K H^i(\Bx; N)$ is
a finite number.  In general we do not know any relation between $\chi^c(\Bp, N)$ and $\chi^c(\Bx, N)$. The purpose of this paper is to describe a relation between the two Euler characteristics for a large class of holonomic modules.

\s \label{hypo} Now consider the case when $R$ is standard graded with $\deg X_i = 1$ for all $i$. Note the Weyl algebra $A_{n+1}(K)$ is also graded with $\deg \partial_i = -1$ for all $i$.
Let $I$ be a graded ideal in $R$. For $i \geq 0$
let $H^i_I(R)$ be the $i^{th}$-local cohomology module of $R$ with
respect to $I$. Let $M = H^i_I(R)$ for some $i$. By a result due
to Lyubeznik, see \cite{L}, $M$
is a finitely generated graded $A_n(K)$-module. In fact
$M$ is a \textit{holonomic} $A_{n+1}(K)$ module. We prove
\begin{theorem}\label{main}
(with hypotheses as in \ref{hypo}). We have
$$\chi^c(\Bp, M) = (-1)^{n+1}\chi^c(\Bx, M).$$
\end{theorem}
In fact we prove  Theorem \ref{main} for a more general class of graded holonomic modules known as generalized Eulerian holonomic modules, see \ref{gen-e} for its definition.
The main technical result we prove is
\begin{theorem}\label{main-lemm}
(with hypotheses as in \ref{hypo}). Let $N$ be a graded, holonomic, generalized Eulerian $A_{n+1}(K)$-module. If $N = N_{X_0}$ then
$$\chi^c(\Bp, N) = 0.$$
\end{theorem}
We note that if $N = N_{X_0}$ then it is easy to note that the Koszul cohomology modules $H^i(\Bx, N)$ vanish for all $i$, see \ref{chi-X}.

\begin{remark}
  Although we stated all the results for cohomology modules, we will work with de Rham homology and Koszul homology.
\end{remark}

\begin{remark}
  The main contribution of the paper was to guess the result. The proofs are not hard.
\end{remark}
We now describe in brief the contents of the paper. In section two we discuss a few preliminary results that we need.
In section three we prove Theorem \ref{main-lemm}. Finally in section four we prove Theorem \ref{main}.
\section{Preliminaries}
In this section we discuss few preliminary results that we need.

Let $K$ be a field of characteristic zero and  let $R=K[X_0, X_1, \ldots, X_n]$  with standard grading. Let $A_{n+1}(K)$ be the $(n + 1)^{th}$ Weyl algebra over $K$. Set $\deg \partial_i= -1$. So $A_{n+1}(K)$ is a graded ring with $R$ as a graded subring.
The Euler operator, denoted by $\epsilon$, is defined as
$$\epsilon := \sum_{i=0}^n X_i\partial_i. $$

Note that $\deg \epsilon =0$. Let $M$ be a graded $A_{n + 1}(K)$-module. If $m\in M$ is homogeneous element, set $|m|= \deg m$.
\begin{definition}
 Let $M$ be a graded $A_{n+1}(K)$-module. We say $M$ is an \textit{Eulerian} $A_{n + 1}(K)$-module if for any homogeneous $z$ in $M$
 we have
$$ \epsilon z= |z|\cdot z.$$
\end{definition}
Clearly $R$ is an Eulerain $A_{n + 1}(K)$-module. Also by \cite[5.3]{MaZhang} graded local cohomology modules $H^i_I(R)$ is Eulerian.
Eulerian modules do not behave well under extensions, see \cite[3.6(1)]{MaZhang}. To rectify this the first author introduced the following notion:
\begin{definition}\label{gen-e}
A graded $A_{n + 1}(K)$-module $M$ is said to be \textit{generalized Eulerian} if for any homogeneous element $z$ of $M$
there exists a positive integer $a$ (depending on $z$) such that
$$ (\epsilon- {|z|})^a \cdot z = 0.$$
\end{definition}
\s We will need the following two computations. Let $E(K)$ be the injective hull of $K$ as an $R$-module. Then note $E(K)$ is an $A_{n+1}(K)$-module. We have
\begin{align*}
  H_i(\Bx, E(K))=
 \begin{cases}
    K \quad \text{if} \ i=n+1.\\
  0 \quad \text{if} \ i\neq n+1.
 \end{cases}
 \end{align*}
Also
\begin{align*}
  H_i(\Bp, E(K))=
 \begin{cases}
    K \quad \text{if} \ i=0.\\
  0 \quad \text{if} \ i\neq 0.
 \end{cases}
 \end{align*}
 The second computation follows from \cite[2.3]{TJ}. For the first note that $X_n \colon E_R(K) \rt E_R(K)$ is surjective with kernel $E_S(K)$ where $S = K[X_0, X_1, \ldots, X_{n-1}]$. An easy induction deduces the result.

 We will need the following result which is certainly known. We give a proof for the convenience of the reader.
 \begin{proposition}
 \label{chi-X} Let $R= K[X_0, X_1, \ldots, X_n]$ and let $M$ be an $R$-module. If $M = M_{X_0}$ then the Koszul homology modules $H_i(\Bx, M) = 0$ for all $i$. In particular $\chi(\Bx, M) = 0$.
 \end{proposition}
 \begin{proof}
   As $M = M_{X_0}$ it follows that multiplication by $X_0$ is bijective on $H_i(\Bx, M)$. But $X_0H_i(\Bx, M) = 0$, see \cite[1.6.5]{BH}. The result follows.
 \end{proof}
\section{Euler characteristics of \MakeLowercase{de} Rham homology}

\s \label{setup}
 Throughout this section $R$ and $M$ denote as in the following:
 Let $K$ be a field of characteristic zero and  let $R=K[X_0, X_1, \cdots, X_n]$  be  the polynomial ring in $(n+1)-$variables. Give standard grading to $R$. Let
 $A_{n+1}(K)=D(R,K)=K<X_0, \cdots,X_n, \partial_0, \cdots, \partial_n>$  be the $(n+1)^{th}$ Wey-algebra over $K$.  Consider $A_{n+1}(K)$ graded with $\deg(\partial_i)=-1,$ for $i=0, 1, \cdots, n.$ Let $M=\bigoplus_{n\in \mathbb{Z}}M_n$
 be a graded holonomic generalized Eulerian $A_{n+1}(K)-$module.

In this section we prove
\begin{theorem}\label{mthm}
 (with hypotheses as in \ref{setup}). If $M\cong M_{X_0}$ then $\chi(\Bp, M)=0.$
\end{theorem}
The idea to prove the result is to prove it first when $n = 0$ and then deduce the general result from $n = 0$ case.
\begin{lemma}\label{mlemma1}
 Let  $R=K[X], \ A_1(K)=K<X, \partial>,$ and $\deg (X) =1,$ $\deg(\partial )=-1.$ Then  $H_1(\partial, R_X) = K$ and $H_0(\partial, R_X) =K\frac{1}{X}.$ In particular $\chi(\partial, R_X) = 0.$

\end{lemma}

\begin{proof}
 Let $u \in  H_1(\partial, R_X).$ Suppose $u=a/X^n$ with $a\in R, \ n\geq1$ and $X \not \backslash a.$

 As
 \[0=\partial u = \frac{-1}{n}\frac{a}{X^{n+1}} + \frac{1}{X^n}\partial(a).\]
 So $nX\partial(a) = a. $ Thus $X/ a.$ This is a  contradiction.

 So $u=p(x)$ for $p(x)\in R.$ $\partial(u) =0 \Rightarrow p(x)=c$ (a constant). Therefore $H_1(\partial, R_X) =K.$

 Now we compute $H_0(\partial, R_X).$ Notice that if $p(x) \in R$ then $q =\int p(x) \in R$ and $\partial(q) =p(x).$ Also for $n\geq 2$
 \[\partial \left(\frac{1}{X^{n-1}} \right) = \frac{-(n-1)}{X^n}. \ \text{So} \ \frac{1}{X^n} = \partial \left(\frac{-1}{(n-1)X^n} \right).\]
Thus $H_0(\partial, R_X) \subseteq K \frac{1}{X}.$

Also $\frac{1}{X} \not = \partial(r)$ for any $r\in R_X$. To see this note if
\begin{align*}
 r &= Xp(x)+c_0+\frac{a_1}{X}+\frac{a_2}{X^2}+\cdots + \frac{a_s}{X^s}, \ a_i\in K\\
\text{then} \  \partial(r)& = \partial(Xp(x)) + 0 - \frac{a_1}{X^2}-\frac{\partial a_2}{X^3}- \cdots - \frac{ra_r}{X^{r+1}}\\
 & \not = \frac{1}{X}.
\end{align*}
Therefore $H_0(\partial, R_X) = K\frac{1}{X}.$ Hence $\chi(\partial, R_X) =0.$
\end{proof}

\begin{lemma}\label{mlemma2}
 Let  $R=K[X], \ A_1(K)=K<X, \partial>,$ and $\deg (X) =1,$ $\deg(\partial )=-1$ and $M$ be a  holonomic, generalized Eulerian $A_1(K)-$module. Assume $M=M_X.$ Then
 \begin{enumerate}[\rm (1)]
   \item $H_1(\partial, M) \neq 0.$
   \item There exists a sequence of holonomic generalized Eulerian $A_1(K)$-modules
   $$ 0 \rt R_X^s \rt M \rt U \rt 0,$$
   where $s = \dim H_1(\partial, M)$ and $U = U_X$.
 \end{enumerate}
\end{lemma}
\begin{proof} (1) Let $\epsilon = X\partial.$ As $M$ is generalized Eulerian it follows that if $m \in M$ is homogeneous then
 \[(\epsilon -\deg m)^rm=0,  \  \  \text{for some} \ r\geq 1 \ \text{(depending on $m$)}.\]
Let $m\in M,$ $m\not =0$ and $\deg(m) =0.$ Assume $\epsilon^rm=0$ and $\epsilon^{r-1}m \not =0.$

Suppose if possible $H_1(\partial, M)=0.$ Then
\[\partial \epsilon^{r-1}m \not =0, \ \text{and} \ \partial \epsilon^{r-1}m\in M_{-1} = \frac{M_0}{X}. \]
Thus $\partial \epsilon^{r-1}m = \alpha/X.$ So $X\partial \epsilon^{r-1}m =\alpha.$ Thus $\alpha =  \epsilon^{r}m =0.$
This is a contradiction.  So $H_1(\partial, M) \neq 0$.

(2) Let $s = \dim_K H_1(\partial, M)$. By (1) we have $s \neq 0$. Assume $H_1(\partial, M) = K\alpha_1+ K\alpha_2+\cdots + K\alpha_s.$ Note that $\deg(\alpha_i) =-1$ see \cite[3.5]{TJG}. We have $H_1(\partial, M) \subset M(+1).$ So $\alpha_i$ can be considered as an element of $M_0.$

Now consider $\psi : R^s\rightarrow M$ where $\psi(e_i) = \alpha_i.$ Clearly $\psi $ is $R$ linear. As
\[\partial(r\alpha_i) = \partial(r) \alpha_i+r\partial(\alpha_i) = \partial(r)\alpha_i,\]
 we get that $\psi$ is $A_1(K)$-linear.   Thus $\psi $ is map of generalized Eulerian\\ $A_1(K)-$modules.

 Claim: $\psi$ is injective. \\
 Suppose if possible there exists $v = r_1e_1+ \cdots + r_s e_s \in \ker \psi$ and $v \neq 0$. Then we have
 $r_1\alpha_1 + \cdots + r_s\alpha_s = 0$. Assume without any loss of generality that $t = \deg r_1 = \max\{ \deg r_i \mid 1 \leq i \leq s \}$.
 Then we have
 $$ \partial^{t}(r_1)\alpha_1 + \cdots + \partial^{t}(r_s)\alpha_s = 0.$$
 We note that $\partial^t(r_i) \in K$ and $\partial^{t}(r_1) \neq 0$. It follows that $\alpha_1, \ldots, \alpha_s$ is NOT linearly independent over $K$. This is a contradiction.
 So $\ker \psi = 0$.

 As $\psi$ is injective we have a short exact sequence of generalized Eulerian holonomic modules
\[0\rightarrow R^s\xrightarrow {\bar{\psi}}M\rightarrow V\rightarrow 0.\]
Localising we get
\[0\rightarrow R^s_X\rightarrow M_X=M \rightarrow V_X\rightarrow 0.\]
Set $U = V_X$. The result follows.
\end{proof}

\begin{theorem}\label{mthm1}
Let  $R=K[X], \ A_1(K)=K<X, \partial>,$ and $\deg (X) =1,$ $\deg(\partial )=-1$ and $M$ be a  holonomic, generalized Eulerian $A_1(K)-$module.
If $M\cong M_X$ then $\chi(\partial, M) = 0.$
\end{theorem}
\begin{proof}
 Suppose if possible there is $M$ such that $M\cong M_X$ and $\chi(\partial, M)\neq0.$ Choose such an $M$ with least length. By \ref{mlemma2}, $H_1(\partial, M) \neq 0.$ Also we have an exact sequence of generalized Eulerian modules
 \[0\rightarrow R_X^s \rightarrow M\rightarrow U\rightarrow 0,\]
 with $s = \dim_K H_1(\partial, M) \geq 1$ and $U = U_X$. Note that
 \begin{align*}
  \chi(\partial,M) = & \chi(\partial, R_X^s)+ \chi(\partial, U)\\
   =& 0+\chi(\partial, U).\\
   \text{So} \ \chi(\partial, U) &\neq0\\
   \text{But} \ \lambda(U) =& \lambda(M) -s\lambda(R_X)\\
   < & \lambda(M).
 \end{align*}
This is a  contradiction. Therefore $\chi(\partial, M)=0.$
\end{proof}

As a corollary we have the following result.
\begin{corollary}\label{cor-main}
  Let  $R=K[X], \ A_1(K)=K<X, \partial>,$ and $\deg (X) =1,$ $\deg(\partial )=-1$ and $M$ be a  holonomic $A_1(K)-$module. Assume $M=M_X$. If $M(a)$ is generalized Eulerian for some $a \in \Z$  then $\chi(\partial, M)=0.$
\end{corollary}
\begin{proof}
 As $M(a) = M_X(a)=M(a)_X.$ So $\chi(\partial, M(a)) =0.$ But $\chi(\partial, M(a)) = \chi(\partial, M).$ Thus $\chi(\partial, M)=0.$
\end{proof}

We now give
\begin{proof}[Proof of Theorem \ref{mthm}]
The case when $n = 0$ is already done. So assume $n \geq 1$.
 Note that $M(-1)\xrightarrow {X_0} M$ is an isomorphism and $\partial_1, \cdots \partial_n$ commutes with $X_0.$ So
 \[H_i(\partial^{'} ,M)(-1)\xrightarrow {X_0} H_i(\partial^{'} ,M),\]
 is an isomorphism, where $\partial^{'} = \partial_1,  \cdots \partial_n.$ So $H_i(\partial^{'} , M) =H_i(\partial^{'},M)_{X_0}$ for all $i$.  Also $H_i(\partial^{'} , M)(-n)$ is holonomic, generalized Eulerian $K<X_0, \partial_0>$ module, see \cite[3.4]{TJG}. \\
 So $\chi(\partial_0, H_i(\partial^{'} , M)) =0.$ Thus
 \[\dim_K H_1(\partial_0, H_i(\partial^{'} , M)) = \dim_K H_0(\partial_0, H_i(\partial^{'} , M)).\]
 From the exact sequence (for all $i$)
 \[0\mapsto H_0(\partial_0, H_i(\partial^{'} , M))\mapsto H_i(\partial, M)\mapsto H_1(\partial_0, H_{i-1}(\partial^{'}, M))\mapsto 0,\]
 we get $\chi(\Bp, M)=0.$
\end{proof}

\section{Proof of Theorem \ref{main}}
In this section we give a proof of Theorem \ref{main}.  In fact we prove the result for graded, holonomic and generalized Eulerian modules.

\begin{theorem}
 Let $M$ be a graded holonomic generalized Eulerian $A_{n+1}(K)$ module. Then  $\chi(\Bp, M) = (-1)^{n+1}\chi(\Bx, M).$
\end{theorem}
\begin{proof}
 We prove the result by induction on $\dim \Supp(M).$

 Suppose $\dim \Supp(M) = 0.$ Then $M = E_R(K)(n + 1)^s.$ Notice
 \begin{align*}
  H_i(\Bx, E(K))=
 \begin{cases}
    K \quad \text{if} \ i=n+1.\\
  0 \quad \text{if} \ i\neq n+1.
 \end{cases}
 \end{align*}
Also
\begin{align*}
  H_i(\Bp, E(K))=
 \begin{cases}
    K \quad \text{if} \ i=0.\\
  0 \quad \text{if} \ i\neq 0.
 \end{cases}
 \end{align*}
 So the result holds when $\dim\Supp(M)=0.$

 Assume $r\geqq 1$ and result holds for all graded holonomic generalized Eulerian $A_{n+1}(K)$ modules $N$ with $\dim \Supp N \leq r-1.$

 As $\Gamma_{\m}(M)=E(K)^s(n+1)$ for some $s\geq 0,$ and we have an exact sequence
 \[0\longrightarrow \Gamma_{\m}(M)\longrightarrow M\longrightarrow \bar{M}\longrightarrow 0.\]
 Thus it suffices to prove the result for $\bar{M}.$ So we may  assume $\Gamma_{\m}(M) =0.$

 Notice  $(X_0, X_1, \cdots, X_{n})\not \in \Ass(M).$ As $\Ass(M)$ is a finite set and $K$ is infinite there exist
 \[\alpha_0X_0+\alpha_1X_1+\alpha_2X_2+\cdots + \alpha_sX_s \in K^{n+1} \backslash \cup_{P\in \Ass(M)}(P\cap K^{n+1}).\]
 Without loss of generality assume $\alpha_0\not=0.$ So after a homogeneous linear change of co-ordinates we can assume $X_0$ is a non-zero divisor on $M.$

 We have an exact sequence
 \[0 \longrightarrow M \longrightarrow M_{X_0} \longrightarrow H_{X_0}^1(M) \longrightarrow 0,\]
 Notice $\Supp(H_{X_0}^1(M)) \subseteq \Supp(M) \cap V(X_0),$ so $\dim\Supp H_{X_0}^1(M) \leq r-1.$ Thus by induction hypothesis
 \[\chi(\Bp, H_{X_0}^1(M)) =(-1)^{n+1}\chi(\Bx, H_{X_0}^1(M)).\]
 By Theorem  \ref{mthm}, $\chi(\Bp, M_{X_0})=0.$ Also note that $\chi(\Bx, M_{X_0})=0.$ We have
 \[\chi(\partial, M) = -\chi(\partial, H_{X_0}^1(M)).\]
 Similarly
 \[\chi(\Bx, M) = -\chi(\Bx, H_{X_0}^1(M)).\]
 So the result follows.
\end{proof}

\end{document}